\documentclass[10pt]{amsart}
\usepackage{amsmath,amssymb,amsthm,graphicx,a4wide}

\newtheorem{theorem}{Theorem}    
\newtheorem{lemma}{Lemma} 
\newtheorem{claim}{Claim}

\newtheorem{corollary}{Corollary}
\theoremstyle{definition}

\newtheorem{remark}[theorem]{Remark}
\newtheorem*{remark*}{Remark}

\title{A characterization of almost alternating knots}
\author{Tetsuya Ito}
\address{Department of Mathematics, Graduate School of Science, Osaka University \\ 1-1 Machikaneyama Toyonaka, Osaka 560-0043, JAPAN}
\email{tetito@math.sci.osaka-u.ac.jp}
\urladdr{http://www.math.sci.osaka-u.ac.jp/~tetito/}
\thanks{T.I. was partially supported by JSPS Grant-in-Aid for Young Scientists (B) 15K17540.}
 
\begin{document}

\begin{abstract}
Generalizing Howie and Greene's characterization of alternating knots, we give a topological characterization of almost alternating knots.
\end{abstract}

\maketitle

\section{Introduction}

Recently, an intrinsic characterization of alternating knots and links in $S^{3}$ was given by Howie \cite{ho,hothe} and Greene \cite{gr} by using a \emph{spanning surface} of a knot $K$, an embedded connected compact surface $S$ with $\partial S = K$ that is not necessarily orientable.

\begin{theorem}[\cite{ho,gr}]
\label{theorem:altGH}
A knot $K$ in $S^{3}$ is alternating if and only if $K$ has spanning surfaces $B$ and $W$ satisfying one of the following properties.
\begin{enumerate}
\item [(1)] (Howie's characterization \cite{ho})\\ 
$\chi(B)+\chi(W)+\frac{1}{2}|\partial B, \partial W| = 2$. Here $|\partial B, \partial W|$ denotes the distance of two boundary slopes defined by $B$ and $W$.\item [(2)] (Greene's characterization \cite{gr})\\
The Gordon-Litherland pairing of $B$ and $W$ are positive and negative definite, respectively.
\end{enumerate}
\end{theorem}

For a spanning surface $S$, \emph{the Gordon-Litherland pairing} $\langle \;,\; \rangle_{S}$ 
is the symmetric bilinear form on $H_{1}(S)$ defined by $\langle a,b \rangle_{S} = lk(\alpha,p_{S}^{-1}(\beta))$. Here $p_{S}: \nu(S) \rightarrow S$ is the double covering from the unit normal $\nu(S)$ of $S$, and $\alpha$ and $\beta$ denote (multi) curves on $S$ that represent $a$ and $b$. Let $e(S)=-\langle [K],[K]\rangle$ be the \emph{euler number} of $S$, which is equal to the twice of the boundary slope of $S$. Since $\sigma(K)=\sigma(S)+\frac{1}{2}e(S)$ \cite[Corollary 5]{gl}, where $\sigma(K)$ denotes the signature of $K$, two characterizations are essentially the same although they look different at first glance.

Inspired from Howie and Greene's argument, we explore a similar characterization for almost alternating knots. A knot $K$ is \emph{almost alternating} if $K$ is represented by an \emph{almost alternating diagram}, a diagram such that a single crossing change makes the diagram alternating (see \cite{a+} for its basic properties). 
In this paper we regard an alternating knot as a special case of an almost alternating knot.

For a spanning surface $S$, we say that an embedded disk $D$ in $S^{3}$ is an \emph{almost compressing disk} of $S$, if $D$ has the following properties.
\begin{enumerate}
\item The knot $K=\partial S$ transversely intersect with the interior of $D$ at one point.
\item $D$ transversely intersects with $S$, except one point $p_{D} \in \partial D$. At the point $p_{D}$, $D$ has a saddle tangency with $S$.  
\item The intersection $D \cap S$ is a union of $\partial D$ and a simple arc $\gamma_{D}$ connecting $p_{D}$ and the unique intersection point $K\cap D$. We call this arc the \emph{intersection arc} of $D$ (See Figure \ref{fig:acd} (a)).
\end{enumerate}

\begin{figure}[htb]
\begin{center}
\includegraphics*[bb=76 613 337 737, width=80mm]{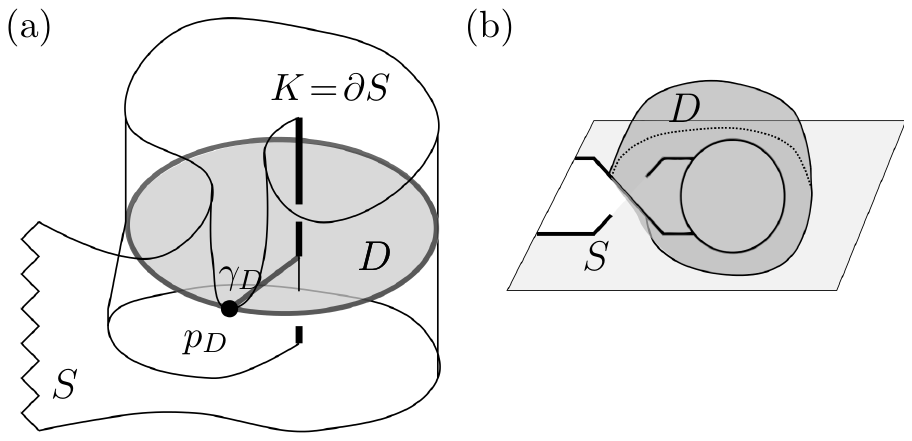}
\caption{Almost compressing disk for spanning surface}
\label{fig:acd}
\end{center}
\end{figure}

A typical situation where an almost compressing disk appears is a checkerboard surface of a knot diagram with a reducible crossing.
A crossing $c$ in a knot diagram is a \emph{reducible crossing} if there is a circle $\delta$ in the projection plane which transversely intersects the diagram at one point $c$. Let $D$ be the disk bounded by such a circle $\delta$ lying in the upper half space. After a slight perturbation near $c$, $D$ gives an almost compressing disk of a checkerboard surface of the diagram (See Figure \ref{fig:acd} (b)).

Using spanning surfaces, their Gordon-Litherland pairings and almost compressing disks, our characterization of almost alternating knots is stated as follows. (Although our characterization can be generalized for links (See Remark \ref{remark:link}), throughout the paper we will mainly treat a knot case for sake of simplicity.)

\begin{theorem}
\label{theorem:mainaa}
A knot $K$ in $S^{3}$ is almost alternating if and only if $K$ has spanning surfaces $B$ and $W$ which intersect transversely, such that
\begin{enumerate}
\item[(i)] $b_{1}(W)+b_{1}(B)-|\sigma(W)-\sigma(B)| \leq 2.$
\item[(ii)] There exist an almost compressing disk $D_{B}$ of $B$ and an almost compressing disk $D_{W}$ of $W$ such that 
\begin{enumerate}
\item[(ii-a)] $D_{B} \cap D_{W}$ transversely intersects at exactly one clasp intersection.
\item[(ii-b)] The clasp intersection $D_{B} \cap D_{W}$  is equal to $\gamma_{D_B} \cap \gamma_{D_W}$.
\item[(ii-c)] The union of intersection arc $\gamma_{D_B} \cup \gamma_{D_W}$ is contained in $B \cap W$.
\end{enumerate} 
\end{enumerate}
\end{theorem}

One can understand the condition (ii) as follows. See Figure \ref{fig:claspacdisk} (a) for an illustration of almost compressing disks satisfying the condition (ii-a) and (ii-b). Near the clasp intersection $D_{B} \cap D_W = \gamma_{D_B} \cap \gamma_{D_{W}}$, two spanning surfaces $B$ and $W$ appear so that the condition (ii-c) is satisfied. Figure \ref{fig:claspacdisk} (b) gives a local model for surfaces $B,W,D_B$ and $D_W$ that satisfy the condition (ii).

The spanning surfaces $B$ and $W$ are mutually intersecting twisted bands. Their almost compressing disks $D_{B}$ and $D_{W}$ appear as half planes, slightly perturbed so that they form a clasp intersection.

To understand Figure \ref{fig:claspacdisk} (b), we take a local coordinate and consider a sequence of slices by the horizontal planes, as shown in Figure \ref{fig:claspacdisk} (c). As the height $t$ increases, two points which are the the slice of the knot $K$ turn. The slice of spanning surfaces $B$ and $W$ turn accordingly. At the critical moment $t=0$, the intersection of $B$ and $W$ appears as a vertical line segment. In the sequence of slice, $D_{B}$ appears as a family of vertical half lines. At $t=0$, $B$ overlaps with $D_B$ and the intersection arc $\gamma_{D_B}$ appears. Also, $D_{W}$ appears as a half-plane in the critical level $t=0$.
By chasing the movie of slices, we see that at $t=0$, $D_{B}$ and $D_{W}$ forms a clasp intersection which is equal to $\gamma_{D_B}\cap \gamma_{D_W}$. Moreover, in the slice $t=0$, the union of intersection arcs $\gamma_{D_{B}} \cup \gamma_{D_W}$ coincides with the vertical line segment $B \cap W$.

\begin{figure}[htbp]
\begin{center}
\includegraphics*[bb= 80 439 429 734, width=120mm]{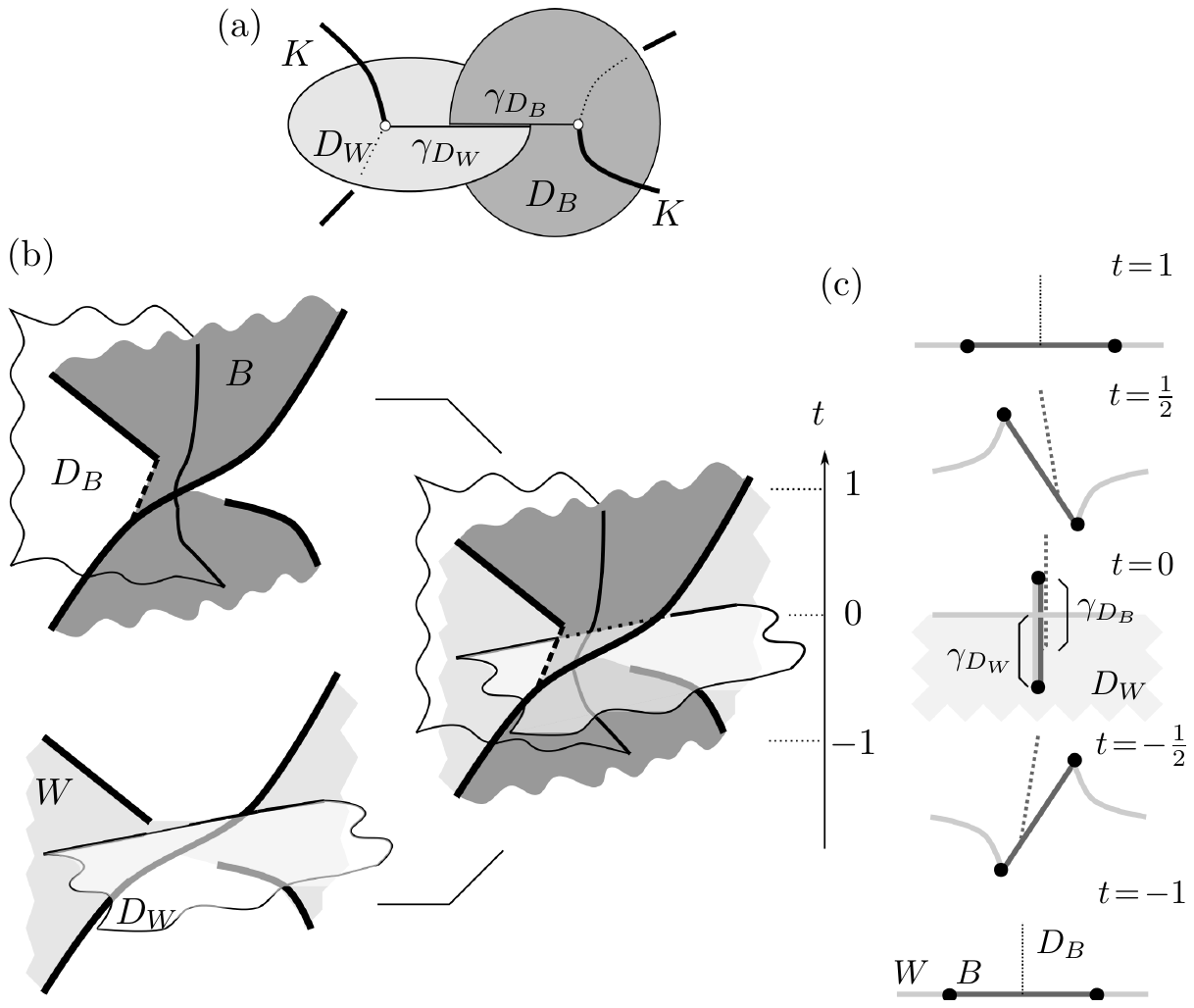}
\caption{(a) Almost compressing disk in condition (ii) of Theorem \ref{theorem:mainaa}. (b) Visualization of $B,W,D_{B}$ and $D_{W}$ in condition (ii) of Theorem \ref{theorem:mainaa}. (c) Movie of slices}
\label{fig:claspacdisk}
\end{center}
\end{figure}

These almost compressing disks and spanning surfaces appear from an alternating diagram on the torus coming from an almost alternating diagram. For an almost alternating diagram $D$ we add a one-handle near the almost alternating crossing $c$. By pushing the almost alternating crossing $c$ to the one-handle, we reverse the over-under information at $c$ to get an alternating diagram on the standardly embedded torus $T$ (see Figure \ref{fig:aadia2} (a,b)).

Let $B$ and $W$ be the checkerboard surface from the resulting alternating diagram on the torus $T$. Take a meridian a (the boundary of a co-core of $H$) and longitude of the torus $T$ so that they intersect exactly once at the crossing point $c$. The disks bounded by these curves give rise to an almost compressing disk $D_{B}$ and $D_{W}$. Near the common point $c$ we slightly push $D_{B}$ and $D_{W}$ so that they form a clasp intersection (see Figure \ref{fig:aadia2} (c)). Then the surfaces $B,W,D_B$ and $D_W$ near $c$ are the same as our local model in Figure \ref{fig:claspacdisk} (b) so they satisfy the condition (ii) in Theorem \ref{theorem:mainaa}.

\begin{figure}[htbp]
\begin{center}
\includegraphics*[bb=80 594 394 738, width=100mm]{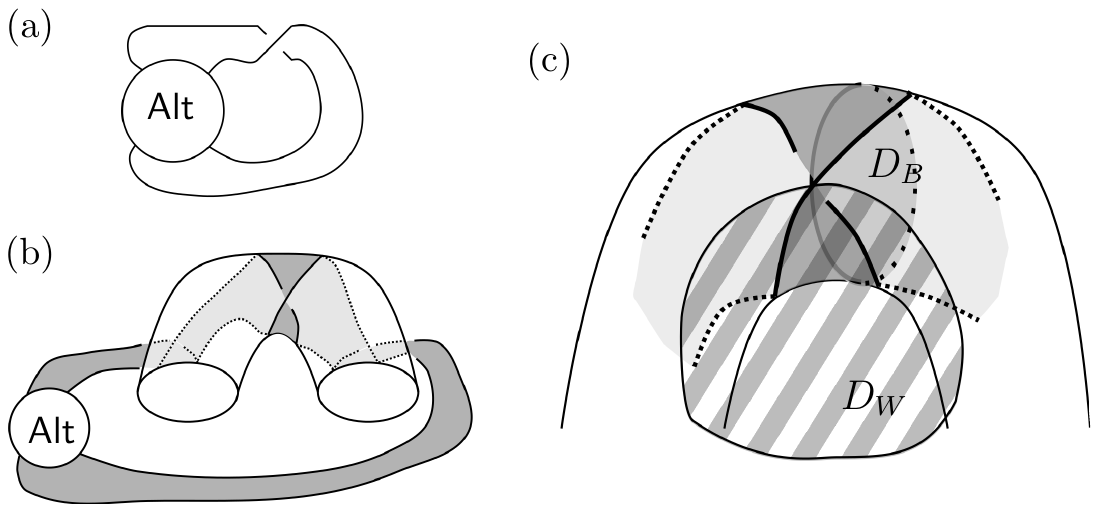}
\caption{(a) Almost alternating diagram (b) Alternating diagram on torus coming from almost alternating diagram and its checkerboard surfaces. (c) Almost compressing disks.}
\label{fig:aadia2}
\end{center}
\end{figure}

As these arguments demonstrate, our proof of Theorem \ref{theorem:mainaa} comes from a point of view that an almost alternating knot is a special case of a \emph{toroidally alternating knot}, a knot admitting a cellular alternating projection on a standardly embedded torus \cite{ad}. Here we say that a knot diagram on a surface $\Sigma$ is \emph{cellular} if it cuts $\Sigma$ into a disjoint union of disks. 

By Theorem \ref{theorem:altGH}, in a setting of Theorem \ref{theorem:mainaa} if either $W$ or $B$ is compressible then $K$ is alternating. The almost compressing disk condition (ii) says that both $B$ and $W$ are `close' to compressible and encodes where is an almost alternating crossing in terms of the almost compressing disks. 

Although in this point of view, it is natural to expect the condition (i) in Theorem \ref{theorem:mainaa} is equivalent to toroidally alternating, it is not the case. (See Remark \ref{remark:tor}). Nevertheless, the quantity $b_{1}(W)+b_{1}(B)-|\sigma(W)-\sigma(B)|$ is interesting in its own right. For a knot $K$ we define the \emph{spanning surface defect} $d(K)$ by 
\[ d(K) = \frac{1}{2}\min \{ b_{1}(W)+b_{1}(B)-|\sigma(W)-\sigma(B)| \mid B,W \text{ are spanning surfaces of } K\}. \]
Note that Theorem \ref{theorem:altGH} says that $d(K)=0$ if and only if $K$ is alternating. By definition $d(K\# K') \leq d(K) +d(K')$, hence $d(K)$ is an \emph{alternating distance} (see \cite{lo}), a quantity which measures to what extent a knot is far from alternating.

Let $g_{alt}(K)$ be the \emph{alternating genus} of $K$, the minimum genus of a Heegaard surface $\Sigma$ in $S^{3}$ such that $K$ has a cellular knot diagram on $\Sigma$. We will show the following inequality.

\begin{theorem}
\label{theorem:estimate}
$d(K) \leq g_{alt}(K)$.
\end{theorem}

It is interesting that our proof of Theorem \ref{theorem:estimate} is inspired from Greene's argument based on Gordon-Litherland pairings, whereas the proof of Theorem \ref{theorem:mainaa} is more related to Howie's geometric argument.

\section*{Acknowledgements}
The author was partially supported by JSPS KAKENHI, Grant Number 15K17540. 
He gratefully thanks to Joshua Howie and Joshua Greene for pointing out an error in the first version. He sincerely wish to thank Joshua Howie for sharing his deep insight. He also thanks to Tetsuya Abe for comments on earlier draft of the paper.

\section{Proofs}
\label{section:ssd}

Let $K$ be a knot which is represented by a diagram $D$ on an oriented embedded closed surface $\Sigma \subset S^{3}$. Assume that the diagram $D$ admits a checkerboard coloring. That is, for each complementary region of the diagram $\Sigma \setminus D$ one can associate the black or white colors so that no two adjacent regions have the same color. 
By attaching twisted bands at the corners of black-colored (resp. white-colored) regions, we get a spanning surface $B$ (resp. $W$) of $K$ which we call the \emph{checkerboard surfaces}.

We say that a crossing $c$ of a diagram $D$ is of \emph{type a} (resp. \emph{type b}) if near the crossing $c$, the coloring is as shown in Figure \ref{fig:types} (1). Also, we orient the diagram $D$ and we say that $c$ is of \emph{type I} (resp. \emph{type II}) if near the crossing $c$, coloring and orientations are as shown in Figure \ref{fig:types} (2).

\begin{figure}[htb]
\begin{center}
\includegraphics*[bb= 74 672 303 740,width=70mm]{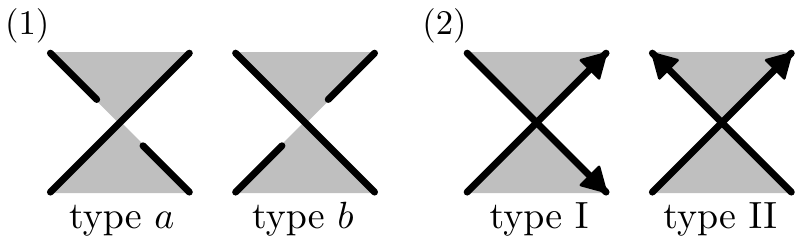}
\caption{Types of crossings. In the definition of type I and type II crossing, over-under information is not used.}
\label{fig:types}
\end{center}
\end{figure}

Let $c(D)$ be the number of the crossings of the diagram $D$.
We denote the number of the crossings of type $a$ and $b$, by $a(D)$ and $b(D)$. respectively.

\begin{lemma}
\label{lemma:generalGL}
Let $B$ and $W$ be the checkerboard surface of a knot $K$ coming from a cellular knot diagram $D$ on a surface $\Sigma$. 
\begin{enumerate}
\item[(i)] At least one of $B$ or $W$ is non-orientable.
\item[(ii)] $b_{1}(W)+b_{1}(B) =  c(D)+b_{1}(\Sigma)$.
\item[(iii)] $b(D)- a(D)= \frac{1}{2}e(B)-\frac{1}{2}e(W) = \sigma(W)-\sigma(B)$.
\end{enumerate}
\end{lemma}
\begin{proof}
(i): Fix a checkerboard coloring of $D$. Then the notion of type I and type II crossings does not depend on a choice of an orientation of $K$. If the diagram has a crossing of type I (resp. type II), then $W$ (resp. $B$) is non-orientable.\\

(ii): A cellular diagram $D$ induces a cellular decomposition of $\Sigma$ whose $0$-, $1$- and $2$-cells correspond to the black colored regions, the crossings, and the white colored regions, respectively. By definition, $b_{1}(W) =c(D)- \#\{\text{White colored regions}\} +1$ and $b_{1}(B) = c(D) - \#\{\text{Black colored regions}\} +1$. Therefore
\begin{eqnarray*}
b_{1}(W)+b_{1}(B)&=&c(D) + 2 - \left( \#\{\text{Black colored regions}\} + c(D) - \#\{\text{White colored regions}\} \right)  \\
& = & c(D) + (2-\chi(\Sigma)) = c(D)+ b_{1}(\Sigma).\\
\end{eqnarray*}

(iii): Let $f_{\Sigma}$ be the blackboard framing of $K$, the framing determined by $\Sigma$. By comparing the framings determined by the checkerboard surfaces and the blackboard framing (see Figure \ref{fig:framing}), we get 
\[   \frac{1}{2}e(B) - f_{\Sigma}= b_{I}(D) - a_{I}(D), \quad \frac{1}{2}e(W)-f_{\Sigma} = a_{II}(D)-b_{II}(D). \]
Here $a_{I}(D)$ denotes the number of crossings which is both of type a and of type I. The meanings of $a_{II}(D),b_{I}(D)$ and $b_{II}(D)$ are similar.
Therefore
\[ \frac{1}{2} e(B)- \frac{1}{2}e(W) = b(D)-a(D).\]

\begin{figure}[htb]
\begin{center}
\includegraphics*[bb=80 624 337 734,width=70mm]{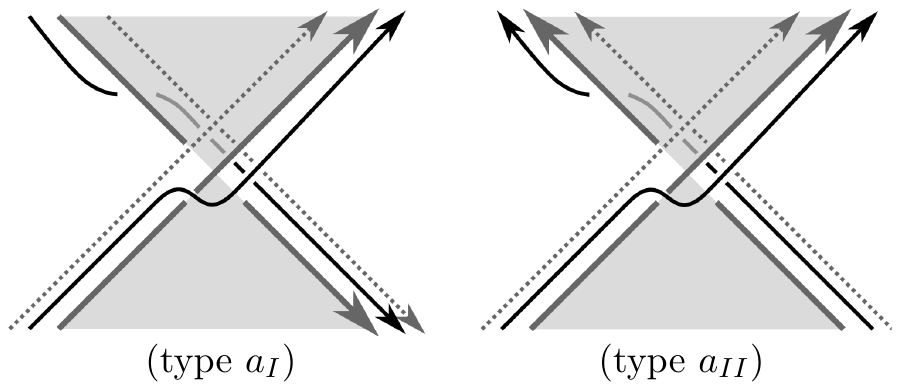}
\caption{Difference of the blackboard framing (dotted arrow) and the framing from $B$ (black arrow), near the crossing of type $a_I$ and $a_{II}$. Crossing of type $b_I$ and $b_{II}$ are similar.}
\label{fig:framing}
\end{center}
\end{figure}
\end{proof}

Lemma \ref{lemma:generalGL} gives an estimate of the \emph{crosscap number} $C(K)$ of a knot $K$, the minimum 1st betti number of a \emph{non-orientable} spanning surface of $K$.
\begin{corollary}
If a knot $K$ is represented by a cellular diagram $D$ on a surface $\Sigma$ which admits a checkerboard coloring, then
$ C(K) \leq \left\lceil \frac{c(D)}{2} \right\rceil + g(\Sigma)$. Here $\left\lceil x \right\rceil$ denotes the ceiling of $x$, the minimum integer which is greater than or equal to $x$. 
\end{corollary}

\begin{remark}
In a case $g(\Sigma)=0$ (a usual knot diagram), a slightly better bound $C(K) \leq \left\lfloor \frac{c(D)}{2} \right\rfloor$ is known \cite{my}. Here $ \left\lfloor x \right\rfloor$ denotes the floor of $x$, the maximum integer which is smaller than or equal to $x$. Indeed, it is easy to see that a slightly better inequality $C(K) \leq \left\lfloor \frac{c(D)}{2} \right\rfloor + g(\Sigma)$ holds unless $c(D)+2g(\Sigma) \equiv 1 \pmod{4}$.
\end{remark}

A cellular diagram may not admit a checkerboard coloring in general. The following is an interesting property of an alternating cellular diagram.

\begin{lemma}
\label{lemma:cc}
If a link diagram $D$ on a surface $\Sigma$ is cellular and alternating, then it admits a checkerboard coloring. 
\end{lemma}
\begin{proof}
Take $2g$ simple closed curves $\gamma_1,\ldots,\gamma_{2g}$ cutting the surface $\Sigma$ into a $2g$-gon $X$ so that they are disjoint from the crossings of $D$ and transverse to $D$. At each $p \in \partial X \cap D$, we assign the symbol $o$ (over) or $u$ (under) according to the over-under information of the arc in $X$ that passes $p$. The alternating assumption of $D$ implies that the symbol $o$ and $u$ alternate along $\partial X$. The cellular assumption of $D$ implies that $\#(\gamma_i,D) \neq 0$ for all $i$.

Assume to the contrary that $D$ admits no checkerboard coloring. If all of  $\#(\gamma_i,D)$ are even, a checkerboard coloring of $X$ extends to a checkerboard coloring of $\Sigma$. Thus we may assume that $\#(\gamma_1,D)$ is odd. Since $\#(\gamma_2,D) \neq 0$, when we glue $X$ along $\partial X$ to recover $\Sigma$, the points with the same symbol is identified (see Figure \ref{fig:cc}). This contradicts the assumption that $D$ is alternating.

\begin{figure}[htb]
\begin{center}
\includegraphics*[width=90mm, bb= 91 648 354 727]{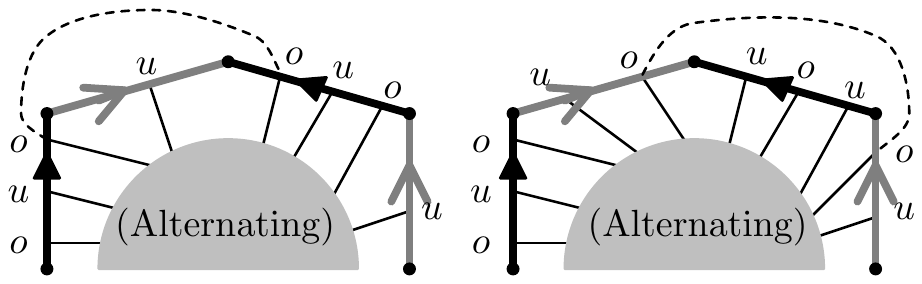}
\caption{Cellular alternating diagram admits a checkerboard coloring. Here a black edge represents $\gamma_1$ and an gray edge represents $\gamma_2$.}
\label{fig:cc}
\end{center}
\end{figure}

\end{proof}

Now we are ready to prove our main theorems.

\begin{proof}[Proof of Theorem \ref{theorem:estimate}]
Assume that $K$ admits a cellular alternating diagram $D$ on an embedded closed surface $\Sigma$ of genus $g$. Let $B$ and $W$ be the checkerboard surface from $D$, whose existence is guaranteed by Lemma \ref{lemma:cc}. We can choose the checkerboard coloring so that every crossing point is of type $b$ because $D$ is alternating. By Lemma \ref{lemma:generalGL} 
\[ 
c(D) = |b(D)- a(D)|= \left|\frac{1}{2}e(B)-\frac{1}{2}e(W)\right| =|\sigma(W)-\sigma(B)| \leq b_{1}(W) +b_{1}(B) = c(D)+b_{1}(\Sigma).  \]
Hence 
\[ b_{1}(W)+b_{1}(B)-|\sigma(W)-\sigma(B)| \leq (c(D)+b_{1}(\Sigma)) -c(D) = 2g. \]
\end{proof}

The proof of Theorem \ref{theorem:estimate} does not use the property that $\Sigma$ is a Heegaard surface. We actually show that $d(K)$ gives a lower bound of the minimum genus of closed embedded surface \emph{which is not necessarily a Heegaard surface} such that $K$ admits a cellular alternating diagram on $\Sigma$.

\begin{proof}[Proof of Theorem \ref{theorem:mainaa}]

First we prove the `only if' part. The property (i) follows from Theorem \ref{theorem:estimate}. As we have already mentioned in Introduction, an almost alternating diagram of $K$ yields a cellular alternating diagram $D$ on the standardly embedded torus, and we have the checkerboard surfaces $B$ and $W$, their almost compressing disks $D_B$ and $D_{W}$ having the property (ii).

We prove the `if' part. By Theorem \ref{theorem:altGH}, if $b_{1}(W)+b_{1}(B)-|\sigma(W)-\sigma(B)| =0$ then $K$ is alternating. Hence throughout the proof, we will assume that

\begin{enumerate}
\item[(D)] $b_{1}(W)+b_{1}(B)-|\sigma(W)-\sigma(B)| \neq 0$.
\end{enumerate}

First we note that this assumption leads to the following.
\begin{claim}
\label{claim:A}
Both $B$ and $W$ are incompressible. 
\end{claim}

\begin{proof}[Proof of Claim \ref{claim:A}]
If $B$ or $W$ is compressible then compression gives new spanning surfaces $B'$ and $W'$ with $0=b_{1}(W')+b_{1}(B')-|\sigma(W')-\sigma(B')|$. This contradicts with (D).
\end{proof}

We fix a tubular neighborhood $N(K)$ of $K$, and let $X_{K}=S^{3}\setminus N(K)$. We put $B_{X}= X_{K} \cap B$ and $\lambda_{B}=\partial N(K) \cap B = \partial B_{X}$. Similarly, we put $W_{X}= X_{K} \cap W$ and $\lambda_{W}=\partial N(K) \cap W = \partial W_{X}$. 
In condition (ii) we are assuming that $B$ and $W$ intersect transversely. With no loss of generality, we will always assume the following additional transversality properties.

\begin{enumerate}
\item[(T1)] $\lambda_{B}$ and $\lambda_{W}$ are simple closed curves on $\partial N(K)$ that intersect efficiently.
\item[(T2)] $B$ intersects with the almost compressing disk $D_{W}$ of $W$ transversely. Similarly, $W$ intersects with the almost compressing disk $D_{B}$ of $B$ transversely.
\end{enumerate} 
Here we say that two curves \emph{intersects efficiently} if they are transverse and attain the minimal geometric intersection.

Our next task is to simplify the intersections of $W \cap D_{B}$ and $B \cap D_{W}$.

\begin{claim}
\label{claim:almostccdisk}
We can put $B,W,D_{B}$ and $D_{W}$ so that $W \cap D_{B} = \gamma_{D_{B}}$ and $B \cap D_{W} = \gamma_{D_{W}}$ hold.
\end{claim}

\begin{proof}[Proof of Claim \ref{claim:almostccdisk}]
By our assumption (ii) and transversality (T2), on the almost compressing disk $D_{B}$, the connected components of $W \cap D_{B}$ are classified into the following three types.
\begin{enumerate}
\item[(a)] A simple closed curve in $D_{B}$.
\item[(b)] An embedded arc connecting two points on $\partial D_B$.
\item[(c)] An arc connecting a point $p_{D_B}=K\cap D_{B}$ and a point on $\partial D_{B}$.
\end{enumerate}

By our assumption (ii-c), $\gamma_{D_{B}} \subset B \cap W \subset W$, so $\gamma_{D_{B}} \subset D_{B}\cap W$. This shows that a component of type (c) is nothing but the intersection arc $\gamma_{D_{B}}$. 

A simple closed curve component $\delta$ of type (a) bounds a disk $\Delta_{\delta}$ in $D_{B} \setminus \gamma_{D_B}$.
Take an innermost $\delta$ so that the interior of $\Delta_{\delta}$ is disjoint from $W$. Since $W$ is incompressible by Claim \ref{claim:A}, $
\delta$ bounds a disk $\Delta'_{\delta}$ in $W$ and $\Delta_{\delta} \cup \Delta'_{\delta}$ is a 2-sphere bounding a 3-ball $M$ in $S^{3} \setminus K$.

We show that $M \cap D_{W} = \emptyset$. Assume to the contrary that $M \cap D_{W} \neq \emptyset$. Then $\partial D_{W} \subset \Delta'_{\delta}$ so it bounds a subdisk $\Delta'' \subset \Delta'_{\delta}$. Then $D_{W} \cup \Delta''$ gives a sphere that transversely intersect with the knot $K$ at exactly one point, which is impossible.

Therefore we can remove $\delta$ by surgerying $W$ along $\Delta_{\delta}$ (see Figure \ref{fig:surgery} (1)), that is, by pushing $\Delta'_{\delta}$ across the 3-ball $M$. Since $D_{W} \cap M = \emptyset$, this surgery does not affect $D_{W}$ so the condition (ii) is still satisfied. This shows that one can remove all the type (a) components of  $W \cap D_{B}$.

Similarly, an arc component $\delta$ of type (b) cuts $D_{B}$ into a smaller disk $\Delta_{\delta}$ that does not contain $p_{D_{B}}$. Take an outermost $\delta$ so that the interior of $\Delta_{\delta}$ is disjoint from $W$. Then we push $W$ along $\Delta_{\delta}$ to remove the intersection $\delta$  (see Figure \ref{fig:surgery} (2)). Since $D_{W} \cap \Delta_{\delta} = \emptyset$, this surgery does not affect  $D_{W}$ so the condition (ii) is still satisfied.

Therefore we conclude that one can remove all the type (a) and (b) components of $W \cap D_{B}$. Thus $W \cap D_{B}$ consisits of a connected component of type (c) so $W \cap D_{B}=\gamma_{D_B}$.

 By the same argument we put $B$ so that $B \cap D_{W} = \gamma_{D_W}$.

\begin{figure}[htb]
\begin{center}
\includegraphics*[bb= 64 609 422 740,width=110mm]{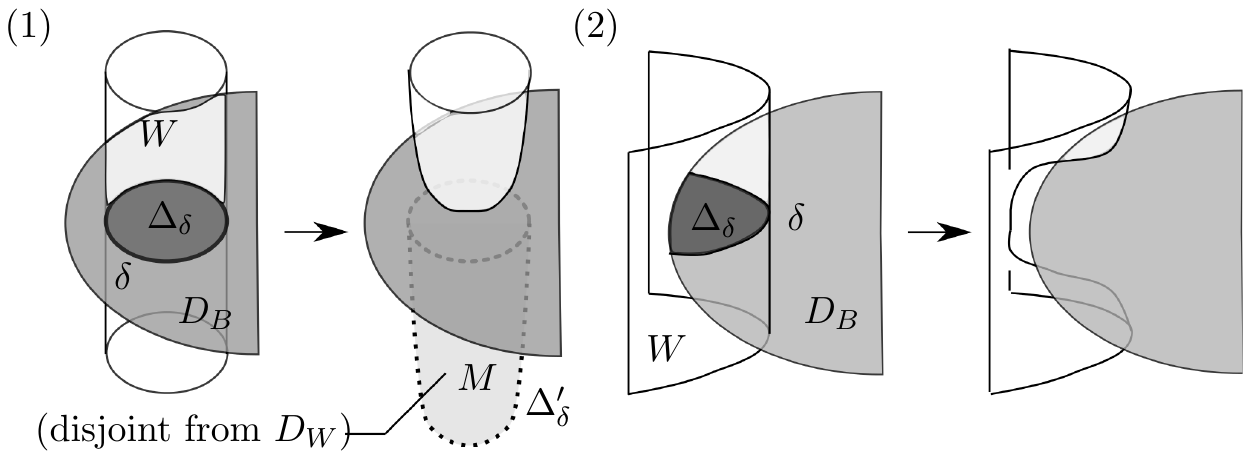}
\caption{Simplifying the intersection $W \cap D_{B}$.}
\label{fig:surgery}
\end{center}
\end{figure}

\end{proof}

Since $B$ and $W$ are embedded, by transversality each connected component of $B_{X}\cap W_{X}$ is either an arc or a simple closed curve. We denote by $\mathcal{A}$ (resp. $\mathcal{C}$) the set of the connected components of $B_{X}\cap W_{X}$ which is an arc (resp. a simple closed curve).

\begin{claim}
\label{claim:essential}
One can put $B$ and $W$ so that no circle $C \in \mathcal{C}$ bounds a disk in $B$ or $W$, preserving the property $W \cap D_{B} = \gamma_{D_{B}}$ and $B \cap D_{W} = \gamma_{D_{W}}$.
\end{claim}
\begin{proof}
Assume to the contrary that a circle $\delta \in \mathcal{C}$ bounds a disk $\Delta_{\delta}$ in $W$ (the case $\delta$ bounds a disk in $B$ is similar). Take an innermost one so that the interior of $\Delta_{\delta}$ is disjoint from $B$. Since $W$ is incompressible by Claim \ref{claim:A}, $C$ bounds a disk $\Delta'_{\delta} \subset W$ and $\Delta_{\delta} \cup \Delta'_{\delta}$ is a 2-sphere bounding a 3-ball $M$ in $S^{3} \setminus K$. By the same argument as Claim \ref{claim:almostccdisk}, $M$ is disjoint from $D_{W}$.
 
By Claim \ref{claim:almostccdisk}, $D_{B} \cap W = \gamma_{D_{B}}$. This shows that $(D_{B} \cap W) \cap X_K$ is contained in an arc in $\mathcal{A}$. Since $\Delta'_{\delta} \cap \gamma =\emptyset$ for any arc $\gamma \in \mathcal{A}$, this shows that $\Delta'_{\delta} \cap D_{B} = \emptyset$. Therefore the 3-ball $M$ is also disjoint from $D_{B}$.

Thus by surgerying $W$ along $\Delta_{\delta}$ (by pushing $\Delta'_{\delta}$ across the 3-ball $M$, we can remove the intersection circle $C$, without affecting $W\cap D_{B}$ and $B \cap D_{W}$. 
\end{proof}

Each arc component $\gamma \in \mathcal{A}$ of $B_{X}\cap W_{X}$ naturally extends to an arc in $B\cap X$ connecting two different points in the knot $K$ which we denote by $\overline{\gamma}$. By \cite[Lemma 1]{ho}, such an intersection arc is `standard', in the following sense. 
\begin{itemize}
\item[(St)]
Each arc $\overline{\gamma}$ is locally homeomorphic to the intersections of two checkerboard surface near the famous ``Menasco crossing ball'' (see Figure \ref{fig:crossball}.)
\end{itemize}

\begin{figure}[htb]
\begin{center}
\includegraphics*[bb= 84 624 337 725,width=70mm]{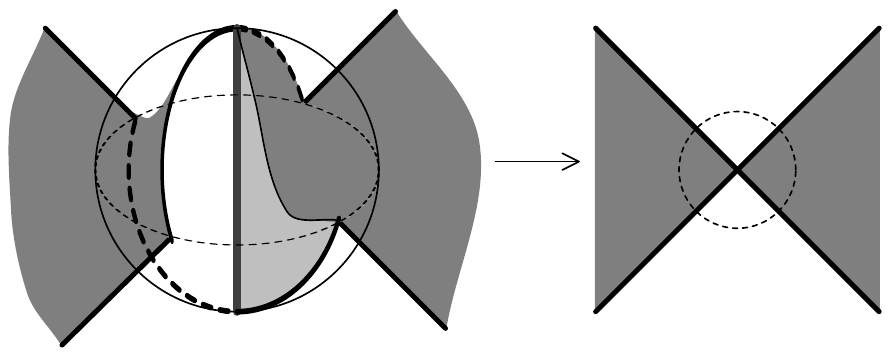}
\caption{Local model of intersection of $B$ and $W$ -- contracting intersection $\overline{\gamma}$ we get an immersed surface.}
\label{fig:crossball}
\end{center}
\end{figure}

Thus by collapsing arcs $\overline{\gamma}$ to a point, we get an immersed surface $\iota: F \rightarrow S^{3}$.

\begin{claim}
\label{claim:achi}
$\chi(F)=0$.
\end{claim}
\begin{proof}[Proof of \ref{claim:achi}]
By the transversality assumption (T1) and the definition of the euler number, $\#(\lambda_{B}\cap \lambda_{W}) = |\frac{1}{2}e(W) -\frac{1}{2}e(B)|$. Therefore 
\[ \chi(B\cap W) = \chi\left(K \cup \sum_{\gamma \in \mathcal{A}} \overline{\gamma} \right)= -\sharp(\lambda_{B} \cap \lambda_{W}) =- \left| \frac{1}{2}e(W)-\frac{1}{2}e(B)\right| =- |\sigma(B)-\sigma(W)|.\]
 Hence we conclude
\[ \chi(F) = \chi(B \cup W) = \chi(B) +\chi(W) -\chi(B\cap W) = 2- b_{1}(W)-b_{1}(B) + |\sigma(B)-\sigma(W)|. \]
By our assumption (i) and (D), $ 0 \leq \chi(F) \leq 1$.

On the other hand, for a spanning surface of a knot $K$, the Gordon-Litherland pairing is non-degenerate so $\sigma(S) \equiv b_1(S)$ (mod $2$). Therefore $\chi(F)$ must be even hence $\chi(F)=0$. 
\end{proof}

We denote the collapsing map by $p : B\cup W \rightarrow \iota(F)$ and its extension to ambient space by $\overline{p}:S^{3} \rightarrow S^{3}$. By abuse of notation, we denote the image of double point circles $p(\mathcal{C}) \subset \iota(F)$ by the same symbol $\mathcal{C}$.

Let $G=p(K)$ and $\widetilde{G}:= \iota^{-1}(G)$ be the four-valent graph in $F$.
We assign the black or white color on each complementary region $F \setminus \widetilde{G}$ so that no two adjacent regions have the same color, and that $p^{-1}(B_F) = B$ and $p^{-1}(W_F) = W$ holds. Here $B_{F}, W_{F} \subset F$ denotes the black and white colored regions of $F$.

By assumption (ii-c), $\gamma_{D_B} \cup \gamma_{D_{W}} \subset B\cap W$
so it is identified with $\overline{\lambda}$ for some $\lambda \in \mathcal{A}$. Thus $\gamma_{D_B} \cup \gamma_{D_{W}}$ is sent to a four-valent vertex $v^{*}$ of $G$ under the collapsing map $p$. We consider the simple closed curves $\delta_{B}=p(\partial D_{B}) \subset B_F$, $\delta_{W}=p(\partial D_{W}) \subset W_F$. By Claim \ref{claim:almostccdisk}, they have the following properties.
\begin{itemize}
\item[(B1)] $\delta_{B}$ and $\delta_{W}$ transversely intersect at exactly one point $v^{*}$, and
\item[(B2)] $\delta_{B}$ and $\delta_{W}$ are disjoint from double point circles $\mathcal{C}$.
\end{itemize}

\begin{claim}
\label{claim:emb}
$F$ is an embedded torus.
\end{claim}

\begin{proof}[Proof of Claim \ref{claim:emb}]

By Claim \ref{claim:achi}, $F$ is a torus or a Klein bottle.
Assume to the contrary that $F$ is not embedded. By Claim \ref{claim:essential} every preimage of a double point circle is an essential simple closed curve in $F$. Hence they cut $F$ into annuli or M\"obius bands $A_{1},\ldots, A_{m}$. The graph $\widetilde{G}$ is connected and disjoint from double point circles. Hence it is contained in exactly one component, say $A_{1}$. Thus we may assume that all the other components $A_{2},\ldots,A_{m}$ are contained in a black colored region $B_{F}$. The double point curve is an intersection of $B$ and $W$ so the number of preimage of double point circles in $W_{F}$ and in $B_{F}$ must be the same. This forces to $m=1$ and $A_{1}$ is an annulus.

On the other hand, by (B1) and (B2), $\iota^{-1}(\delta_{B}), \iota^{-1}(\delta_{W})$ are simple closed curves in $A_{1}$ that transversely intersect at exactly one point $v^{*}$. This is a contradiction because an annulus cannot contain such simple closed curves.
 Since the Klein bottle cannot be embedded into $S^{3}$, $F$ is an embedded torus.
\end{proof}

By construction of the surface $F$, the knot $K$ is contained in a neighborhood of $F$. By (T1), the projection of $K$ on $F$ yields an alternating diagram $D$ on $F$. By (B1) and (B2), $\delta_{B}$ and $\delta_{W}$ are simple closed curves which cuts $F$ into a disk. Moreover by Claim \ref{claim:almostccdisk} they bound disks $p(D_{B})$ and $p(D_W)$ in $S^{3}$ whose interiors are disjoint from $F$. Thus $F$ is a standard embedding. Also, the diagram on $F$ is cellular, because otherwise there exists a simple closed curve contained in $B$ or $W$ which is isotopic to $\delta_{W}$ or $\delta_{B}$. Existence of such loop implies $B$ or $W$ is compressible, which contradicts with Claim \ref{claim:A}.

Now by cutting $F$ along $\delta_{B}$ and $\delta_{W}$ we get an alternating tangle in the square, and the diagram on $F$ can be recovered by gluing their edges (see Figure \ref{fig:rdiagram} (a,b)). The resulting alternating diagram on a torus $F$ can be seen as an almost alternating diagram Figure \ref{fig:rdiagram} (c), with almost alternating crossing which corresponds to $v^{*}$.

\begin{figure}[htb]
\begin{center}
\includegraphics*[bb=77 655 393 742,width=110mm]{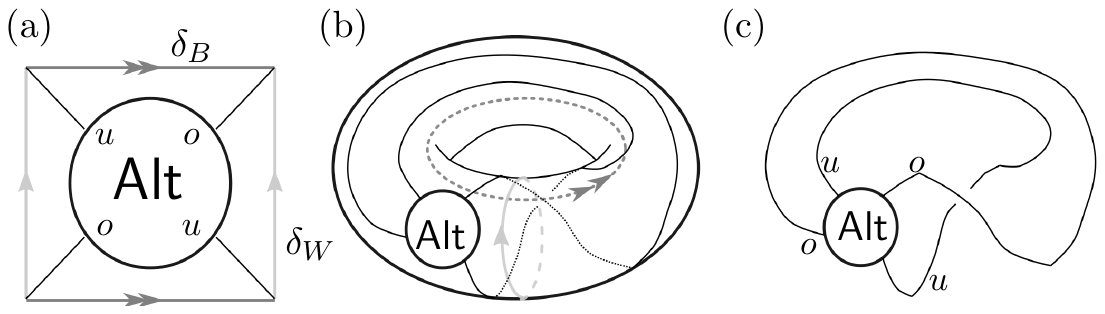}
\caption{(a) Cutting a diagram on $F$ along $\delta_{B}$ and $\delta_{W}$. (b) The resulting alternating diagram on $F$.  (c) Almost alternating diagram from alternating diagram on $F$,.}
\label{fig:rdiagram}
\end{center}
\end{figure}

\end{proof}

\begin{remark}
\label{remark:tor}
As one can see in the proof, we used almost compressing disks not only to control the alternating diagram on torus $F$, but also to prove Claim \ref{claim:emb}, the surface $F$ constructed from spanning surfaces $B$ and $W$ is an embedded torus. In fact, Howie presents a construction of knot with spanning surfaces $B$ and $W$ such that $b_{1}(W)+b_{1}(B)-|\sigma(W)-\sigma(B)| = 2$ and that they do not give rise to an alternating diagram on embedded torus \cite{hothe}. 

Toroidally alternating links and their generalizations, a knot admitting alternating diagram on surface with several additional properties, are extensively discussed in \cite{hothe}. 
\end{remark}

\begin{remark}
\label{remark:link}
In our proof, we used the assumption that $K$ is a knot, only to guarantee the property (St) of the standardness of the intersections of $B$ and $W$. As noted in \cite{ho,hothe}, for a non-split link case, to guarantee the standardness (St) it is sufficient to add an assumption that 
\begin{itemize}
\item[(iii)] The intersection of $\lambda_B$ and $\lambda_{W}$ have the same sign.
\end{itemize}
Thus, by adding the assumption (iii) we have a characterization of almost alternating links as well.
\end{remark}

\end{document}